\documentclass[12pt]{amsart}
\usepackage{amsfonts,amsmath,amsthm,amssymb}
\usepackage{url}
\usepackage{latexsym}
\usepackage{enumerate}
\usepackage{graphics}
\newtheorem{theorem}{Theorem}[section]
\newtheorem{corollary}[theorem]{Corollary}

\newtheorem{lemma}[theorem]{Lemma}
\newtheorem{example}[theorem]{Example}

\newtheorem{proposition}[theorem]{Proposition}

\theoremstyle{definition}

\theoremstyle{remark}
\newtheorem{rem}[theorem]{Remark}
\theoremstyle{remark}

\newcommand{\beql}[1]{\begin{equation}\label{#1}}
\newcommand{\eeq}{\end{equation}}

\begin{document}

\title[All complex Equiangular Tight Frames in dimension 3]{All complex Equiangular Tight Frames in dimension 3}

\author{Ferenc Sz\"oll\H{o}si}

\date{\today.}

\address{F. Sz.: Research Center for Pure and Applied Mathematics, Graduate School of Information Sciences,
Tohoku University, Sendai 980-8579, Japan}\email{szoferi@gmail.com}



\begin{abstract}
In this paper we describe some new algebraic features of the Gram matrices of complex Equiangular Tight Frames (ETF). This lead on the one hand to the nonexistence of several low dimensional complex ETFs; and on the other hand to the full algebraic classification of all complex ETFs in $\mathbb{C}^3$. We use computer aided methods, in particular, Gr\"obner basis computations to obtain these results.
\end{abstract}

\maketitle

	

\section{Introduction}
A finite complex equiangular tight frame (or equiangular Parseval frame; for short, we will use ETF) is a collection of $n$ complex unit vectors $\varphi_1,\varphi_2,\hdots, \varphi_{n}$ in $\mathbb{C}^m$ having mutual inner product (or ``angle'' between distinct vectors $\varphi_i$ and $\varphi_j$) as small as possible in absolute value. In particular, for complex ETFs the Welch bound is attained \cite{JT} and one has
\beql{al}\alpha_{n,m}:=|\left\langle\varphi_i,\varphi_j\right\rangle|=\sqrt{\frac{n-m}{m(n-1)}}\qquad\text{for $1\leq i<j\leq n$}.\eeq
We refer to such objects as $(n,m)$-frames. The case $\alpha_{n,m}=0$ leads to the concept of orthonormal basis and as such is uninteresting in general. Therefore we may assume that $\alpha_{n,m}>0$. Equiangular tight frames arise in many industrial applications, most notably they are used for error correction, see e.g.\ \cite{HP}. It is known that $n\leq m^2$ and in it is conjectured that $(m^2,m)$-frames indeed exist for all $m\geq 1$. This is a notorious open problem in quantum tomography and was attacked by Appleby and Grassl with coauthors recently \cite{AP}, \cite{SG}. The analogous problem in the real case is a question on the existence of certain strongly regular graphs, and goes back to Seidel's seminal paper \cite{S}.

The outline of this paper is as follows. In Section 2 we analyze the Gram matrix of equiangular tight frames and provide some new insight into their structure. We use this to formulate several necessary conditions on the existence of complex equiangular tight frames. In Section 3 we describe these properties by systems of polynomial equations and outline a computer-aided approach towards deciding the existence of complex equiangular tight frames. In Section 4 we illustrate our methods and give a complete algebraic characterization of all complex equiangular tight frames in $\mathbb{C}^3$.

The original contribution of this paper is the proposed novel method, which leads to a mathematically rigorous proof of the nonexistence of complex equiangular $(8,3)$-frames. This is the first nonexistence result of this kind. Another contribution is the complete algebraic classification of all complex equiangular $(9,3)$-frames. This is in turn equivalent to the classification of any of the following objects: tight complex projective $2$-designs with angle set $\{1/4\}$ and $9$ elements \cite{KR}, \cite{Z}; SIC-POVMs of order $3$ \cite{SG}, \cite{Z}; or self-adjoint complex Hadamard matrices of order $9$ with constant diagonal \cite{BN}, \cite{CG}, \cite{FSZframehad}. Our results are similar in spirit to the known classification of Mutually Unbiased Bases in $\mathbb{C}^d$ for $d\leq 5$, see \cite{BBB}.
\section{Gram matrices}
We may identify an equiangular $(n,m)$-frame with an $m\times n$ matrix $F$ whose column vectors are the vectors $\varphi_i$, $1\leq i\leq n$. Since $F$ comes from a tight frame, it follows that all of its rows are pairwise orthogonal and have equal norm \cite{HP}, and therefore the Gram matrix of the frame $G:=F^\ast F$ is self-adjoint, has unit diagonal and all of its off diagonal entries have modulus $\alpha_{n,m}$. Additionally, $G$ satisfies the frame condition
\beql{mg}mG^2=nG,\eeq
and hence, up to a trivial scaling, it is a self-adjoint projection of rank $m$. Note that if $G$ is the Gram matrix of an ETF, then so is $PDGD^\ast P^T$ for any permutation matrices $P$ and for any unitary diagonal matrices $D$. Matrices related in this fashion are called equivalent. This equivalence captures the obvious symmetries of ETFs: rearrangement of the frame vectors and scaling them by an arbitrary complex unimodular phase does not change the fundamental properties of the frame. Consequently one can suppose that apart from the $(1,1)$ position the first row and column of $G$ is $\alpha_{n,m}$.

The following folklore result gives a characterization of the $(n-1)\times(n-1)$ sub-Gram matrices of equiangular tight frames.

\begin{lemma}
Let $n\geq m$ be fixed integers, and let $H$ be a self-adjoint matrix of order $n-1$ with constant diagonal $1$ and off diagonal entries of modulus $\alpha_{n,m}$. Then
\beql{Grep}G=\left[\begin{array}{cc}
H & v\\
v^\ast & 1
\end{array}\right]\eeq
is the Gram matrix of a complex equiangular $(n,m)$-frame, if and only if the entries of $v$ are of modulus $\alpha_{n,m}$, and
\beql{2}nH-mH^2=mvv^\ast.\eeq
\end{lemma}
\begin{proof}
If $G$ is the Gram matrix of an equiangular $(n,m)$-frame, then \eqref{2} follows from equation \eqref{mg}. Conversely, assume that \eqref{2} holds for some column vector $v$ having entries of modulus $\alpha_{n,m}$. Multiply this by $v$ to obtain $nHv-mH^2v=\alpha_{n,m}^2m(n-1)v=(n-m)v$. This implies that the matrix $G$ shown in \eqref{Grep} satisfies \eqref{mg} and therefore it is the Gram matrix of an equiangular $(n,m)$-frame.
\end{proof}
The point is that we have some rather obvious necessary conditions on $H$, and once they are fulfilled, we can reconstruct the missing column vector of the Gram matrix, i.e.\ once we have $n-1$ suitably chosen equiangular lines with common angle $\alpha_{n,m}$, the last vector $\varphi_n$ follows for free. Very little is known about the structure of smaller principal submatrices.
\begin{lemma}\label{projc}
Let $2\leq r \leq n-2$, $G$ be a Gram matrix of an equiangular $(n,m)$-frame and $H$ be its $(n-r)\times (n-r)$ leading principal submatrix. Then the following conditions are met$:$
\begin{gather}\mathrm{rank}\left(mH^2-nH\right)\leq r,\label{12}\\
\left|\frac{n}{m}h_{i,j}-\sum_{k=1}^{n-r}h_{i,k}\overline{h_{j,k}}\right|\leq r\alpha_{n,m}^2,\qquad\text{for all $1\leq i<j\leq n-r$}.\label{13}\end{gather}
\end{lemma}
\begin{proof}
Condition \eqref{12} is obvious, and follows from \eqref{mg}. Condition \eqref{13} follows from \eqref{12} after using that the off-diagonal entries of the Gram matrix are of modulus $\alpha_{n,m}$.
\end{proof}

Condition \eqref{13} is a one-way analytic criterion which can detect inextensible sub-Gram matrices. For the case $r=2$ it is possible to find a stronger, algebraic criterion. We recall a simple, yet extremely useful lemma as follows.
\begin{lemma}[cf.\ \cite{Ha}, \cite{FSZ}]\label{HTIO}
Let $x_1, x_2$, $y_1, y_2$ and $z_1, z_2$ be complex numbers of modulus $1$. Then
\begin{equation}\label{haid}
\begin{split}
\left(x_1\overline{y_1}+x_2\overline{y_2}\right)\left(y_1\overline{z_1}+y_2\overline{z_2}\right)&\left(z_1\overline{x_1}+z_2\overline{x_2}\right)\\
&\equiv\left|x_1\overline{y_1}+x_2\overline{y_2}\right|^2+\left|y_1\overline{z_1}+y_2\overline{z_2}\right|^2+\left|z_1\overline{x_1}+z_2\overline{x_2}\right|^2-4.
\end{split}
\end{equation}
In particular, the left hand side is a real number.
\end{lemma}
\begin{proof}
Follows immediately after observing that $|v|^2=v\overline{v}$ and in particular $\overline{v}\equiv1/v$ for complex numbers $v$ of modulus $1$.
\end{proof}
Haagerup used a variant of Lemma~\ref{HTIO} to give a full classification of all complex Hadamard matrices of order $5$ \cite{Ha}. The (proof of the) next theorem reveals how.
\begin{theorem}[cf.\ \cite{FSZ}]\label{T}
Let $n\geq 5$ and assume that $\left[G\right]_{i,j}=g_{i,j}$, $1\leq i,j\leq n$ is the Gram matrix of a complex equiangular $(n,m)$-frame. Let $1\leq i<j<k\leq n$ be indices. Then, with the notations $\Sigma:=ng_{i,j}/m-\sum_{\ell=1}^{n-2}g_{i,\ell}\overline{g_{j,\ell}}$, $\Delta:=ng_{j,k}/m-\sum_{\ell=1}^{n-2}g_{j,\ell}\overline{g_{k,\ell}}$ and $\Psi:=ng_{k,i}/m-\sum_{\ell=1}^{n-2}g_{k,\ell}\overline{g_{i,\ell}}$, we have
\beql{sdp}\Sigma\Delta\Psi-\alpha_{n,m}^2\left(|\Sigma|^2+|\Delta|^2+|\Psi|^2-4\alpha_{n,m}^4\right)=0,\eeq
 where $\alpha_{n,m}$ is given by formula \eqref{al}.
\end{theorem}
Observe that Theorem \ref{T} describes and algebraic identity with rational coefficients relating the entries of the principal $(n-2)\times (n-2)$ submatrix of $G$.
\begin{proof}
The following idea is essentially due to Haagerup \cite{Ha}. Let $G$ be the Gram matrix of a complex ETF, and consider its $i$th, $j$th and $k$th rows. Then, by using $G^2=n/mG$, we find that $\left\langle G_i,G_j\right\rangle=n/mg_{i,j}$. Using this for the pairs $(i,j)$, $(j,k)$, and $(k,i)$, we get:
\[ng_{x,y}/m-\sum_{\ell=1}^{n-2}g_{x,\ell}\overline{g_{y,\ell}}=g_{x,n-1}\overline{g_{y,n-1}}+g_{x,n}\overline{g_{y,n}}, \qquad \text{for $(x,y)\in\{(i,j), (j,k), (k,i)\}$}.\]
Divide all three equations above by $\alpha_{n,m}^2$ to obtain the sum of two unit vectors on their right hand side. In particular, we can apply Lemma~\ref{HTIO} for the quantities appearing on the left: plugging them into equation \eqref{haid} and scaling by $\alpha_{n,m}^6$ yields the desired result.
\end{proof}
It is unknown whether the conditions described by Theorem \ref{T} are sufficient for reconstructing the last two frame vectors.

Finally, we mention a general necessary condition for the existence of equiangular tight frames. Compared to the local conditions described above, it is a global condition depending on solely the parameters $n$ and $m$.
\begin{theorem}[Naimark, see e.g.\ \cite{NaimarkC}]\label{naimark}
There exists an equiangular complex $(n,m)$-frame if and only if there exists an equiangular complex $(n,n-m)$ frame.
\end{theorem}
\begin{proof}[Proof $($Sketch$)$]
If $G$ is the Gram matrix of a complex equiangular $(n,m)$-frame, then $(nI-mG)/(n-m)$ is the Gram matrix of a complex equiangular $(n,n-m)$ frame.
\end{proof}
Since it is known that the maximum number of equiangular lines is at most $m^2$ in $\mathbb{C}^m$, the theorem above has the following consequence.
\begin{corollary}\label{cor26}
Let $m\geq 3$. There does not exists any complex equiangular $(n,m)$-frame for any $m+2\leq n\leq \left\lceil\frac{1+2m+\sqrt{1+4m}}{2}\right\rceil-1$.
\end{corollary}
In the next section we describe an exact algebraic approach to the search for complex equiangular tight frames. This is in contrast with the widely used and successfully employed numerical methods \cite{JT}.
\section{Polynomial equations and Gr\"obner bases}\label{s3}
In this section we describe the sub-Gram matrices of complex equiangular tight frames as solution sets of a system of polynomial equations. Buchberger was the first, who described an algorithm which can decide if such a system has a common complex solution, by means of computing a so-called Gr\"obner basis \cite{BB}. We do not want to go into the technical details; the interested reader is referred to \cite{BX}, \cite{JJ}, \cite{La} and the references therein.

This algorithm, along with a more efficient variant of it ($F_4$), has been implemented in a number of computer algebra systems. We have used J.-C. Faug\`ere's FGb package\footnote{Version $1.58$ is available at the author's website: \url{http://www-calfor.lip6.fr/~jcf}.} linked to Maple $17$ to do the relevant Gr\"obner basis calculations \cite{JCF}, and double-checked the obtained results with Magma\footnote{Version $2.20$-$3$ for Linux supporting AVX. See \url{http://magma.maths.usyd.edu.au/magma/}.}. In order to reduce the complexity of the problem, we investigated the $(n-2)\times (n-2)$ principal sub-Gram matrices instead of the Gram matrices themselves as outlined in Section $2$. In what follows we enlist the relevant necessary conditions leading to a system of polynomial equations.
\begin{enumerate}[$($a$)$]
\item Choose any $m\geq 2$ and $n\geq m+1$ as desired. These are fixed integer numbers and not variables;
\item Set $H$ as an $(n-2)\times (n-2)$ matrix with entries $h_{i,i}=1$ for all $1\leq i\leq n-2$, $h_{1,i}=h_{i,1}=\alpha$ for all $2\leq i\leq n-2$, $h_{i,j}=\alpha x_{i,j}$ and $h_{j,i}=\alpha/x_{i,j}$ for all $2\leq i<j\leq n-2$. This gives us $\binom{n-3}{2}$ variables;
\item The angle $\alpha$ is determined by the initial parameters $n,m$. In particular, $m(n-1)\alpha^2-(n-m)=0$. If $\alpha_{n,m}$ is irrational, then we consider this equation as a condition on the matrix $H$, otherwise we set $\alpha:=\alpha_{n,m}$ from equation \eqref{al};
\item The rank condition: $\mathrm{rank}(H)\leq m$. In particular, all $(m+1)\times (m+1)$ minors vanish. This describes $\binom{n-2}{m+1}^2$ equations;
\item The frame condition (see Lemma \ref{projc}): $\mathrm{rank}\left(mH^2-nH\right)\leq 2$. This gives rise to another $\binom{n-2}{3}^2$ equations;
\item Haagerup's condition (Theorem \ref{T}): for every triplet of rows $1\leq i<j<k\leq n-2$ the equality \eqref{sdp} must hold. This gives another $\binom{n-2}{3}$ equations;
\item The complex conjugates of the polynomials described in $($f$)$;
\item The nonzero condition: $1-u\prod_{2\leq i<j\leq n-2}x_{i,j}=0$, where $u$ is a ``dummy'' variable.
\end{enumerate}
Conditions $($a$)$-$($h$)$ are those we actually implement. Our desired solution set is further subject to:
\begin{enumerate}[$($a$)$]
\setcounter{enumi}{8}
\item The unimodular condition: $|x_{i,j}|=1$ for all $2\leq i<j\leq n-2$;
\item The nonnegative condition: $\alpha>0$.
\end{enumerate}
In summary, the total number of polynomial equations and variables read
\[\#\text{eqs}=2+\binom{n-2}{m+1}^2+\binom{n-2}{3}^2+2\binom{n-2}{3},\qquad \#\text{vars}=2+\binom{n-3}{2},\]
where the number of variables and the number of equations are both one less in case $\alpha_{n,m}$ is rational. We suspect that some of these conditions are redundant.

We remark here that if a variable $x$ assumes complex unimodular values, then its conjugate is its reciprocal, that is, $\overline{x}=1/x$. Therefore complex conjugation leads to rational functions, whose denominators can be cleared via appropriate scaling by these variables. This implies that, apart from conditions $($i$)$ and $($j$)$, we are indeed dealing with a system of polynomial equations. It is difficult to verify condition $($i$)$ in general (although Cohn's theorem \cite{C} is a fundamental result of interest here. For more consult \cite{FSZ}).

From Corollary \ref{cor26} we know that an equiangular $(5,3)$-frame cannot exist. The following toy example should verify this.
\begin{example}\label{ex35}
Set $m:=3$, $n:=5$, and consider the partial Gram matrix
\[H=\alpha\left[\begin{array}{ccc}
1/\alpha & 1 & 1\\
1 & 1/\alpha & x_{2,3}\\
1 & 1/x_{2,3} & 1/\alpha
\end{array}\right].\]
We have $5$ equations in $3$ variables $\alpha, x_{2,3}$ and $u$, which read, after appropriate scaling, and reduction modulo $6\alpha^2-1$, as follows:
\[\left.\begin{array}{rcc}
6 \alpha^2-1&=&0\\
3 x_{2,3}^4+58 x_{2,3}^3 \alpha+18 x_{2,3}^2+58 x_{2,3} \alpha+3&=&0\\
-22 x_{2,3}^3 \alpha-9 x_{2,3}^2-36 x_{2,3} \alpha-3&=&0\\
-3   x_{2,3}^3-36 x_{2,3}^2 \alpha-9 x_{2,3}-22 \alpha&=&0\\
x_{2,3} u-1&=&0
\end{array}\right\}.\]
By eliminating $x_{2,3}^3$ it is easy to see that this system of equations does not have a solution.
\end{example}
\begin{rem}
For $m=3$ and $n=5$ the following matrix
\[H=\frac{1}{\sqrt6}\left[
\begin{array}{ccc}
 \sqrt{6} & 1 & 1 \\
 1 & \sqrt{6} & a \\
 1 & 1/a & \sqrt{6} \\
\end{array}
\right],\qquad\text{where}\qquad a=-\sqrt{6}/9+5\mathbf{i}\sqrt{3}/9,\]
satisfies all the conditions $($a$)$-$($e$)$ and $($h$)$-$($j$)$ as well as the analytic condition \eqref{13}. This demonstrates that Theorem \ref{T}, and in particular, conditions $($f$)$-$($g$)$ play an essential r\^ole in concluding nonexistence in Example~\ref{ex35}.
\end{rem}
\section{All complex equiangular tight frames in $\mathbb{C}^3$}
In this section we describe all complex equiangular tight frames in $\mathbb{C}^3$. The cases $n=3,4$ are somewhat trivial: in the first case $G$, by \eqref{al}, is the identity matrix. The Gram matrices of the $(4,3)$ frames are in fact real (up to normalization), and the frame vectors form a regular simplex (or tetrahedron). The case $(5,3)$ is impossible, as demonstrated above in Example \ref{ex35}, although one would obtain the same conclusion from the nonexistence of $(5,2)$ tight frames through Naimark's result. The next interesting case is the equiangular $(6,3)$-frames which are in one-to-one correspondence with self-adjoint complex conference matrices of order $6$, see \cite{Ka}, \cite{CG}, \cite{FSZframehad}.
\begin{proposition}[cf.\ \cite{Ka}]
All complex equiangular $(6,3)$-frames correspond to some member of the following one-parameter family of Gram matrices, up to equivalence$:$
\[G_6^{(1)}(a)=\frac{1}{\sqrt5}\left[
\begin{array}{cccc|cc}
 \sqrt{5} & 1 & 1 & 1 & 1 & 1 \\
 1 & \sqrt{5} & a & -a & -1 & 1 \\
 1 & \overline{a} & \sqrt{5} & 1 & -\overline{a} & -1 \\
 1 & -\overline{a} & 1 & \sqrt{5} & \overline{a} & -1 \\
\hline
 1 & -1 & -a & a & \sqrt{5} & 1 \\
 1 & 1 & -1 & -1 & 1 & \sqrt{5} \\
\end{array}
\right].\]
\end{proposition}
\begin{proof}
We use the theory described in Section~\ref{s3}: we have $27$ polynomials in $5$ variables $\alpha, x_{2,3}, x_{2,4}, x_{3,4}$, and $u$. We find, that the ideal $\mathcal{I}$, generated by these polynomials is one dimensional, and in particular $(-1 + x_{2,4}) (1 + x_{2,4}) (-1 + x_{3,4}) (1 + x_{3,4}) (x_{2,4} + x_{3,4})\in\mathcal{I}$. Therefore, by symmetry, it follows that the off-diagonal entries of the normalized Gram matrix are either $\pm1$ or some arbitrary entry, say $a$, and its negative. Via equation \eqref{mg} it is an easy automated calculation to verify that all complex equiangular $(6,3)$-frames belong, up to equivalence, to the family $G_6^{(1)}$ described above.
\end{proof}
Now we turn to the discussion of equiangular $(7,3)$-frames. Such frames can be constructed from skew-symmetric Hadamard matrices, see \cite{RENES}, \cite{FSZframehad}.
\begin{proposition}
There exists a unique complex equiangular $(7,3)$-frame, up to equivalence.
\end{proposition}
\begin{proof}
We have $147$ polynomials in $8$ variables $\alpha, x_{2,3}, \hdots, x_{3,5}$, and $u$. We find, that the ideal $\mathcal{I}$, generated by these polynomials is zero dimensional, and in particular $\{16 + 6 x_{2,3}^2 + 5 x_{2,3}^4 + 6 x_{2,3}^6 + 16 x_{2,3}^8,-9 \alpha x_{2,3}^3-9 \alpha x_{2,3}+4 x_{2,3}^4+3 x_{2,3}^2+4\}\subset\mathcal{I}$. Since, by condition $($j$)$ of Section~\ref{s3}, $\alpha=\alpha_{7,3}=\sqrt{2}/3$, we find that $x_{i,j}\in\{a,\overline{a},a^3,\overline{a}^3\}$, with $a=-\sqrt2/4+\mathbf{i}\sqrt{14}/4$ for all $2\leq i<j\leq 5$. A simple computer search reveals that there are $120$ solutions, which are all equivalent to a single $5\times 5$ matrix. This can be uniquely extended (up to equivalence) to the desired Gram matrix.
\end{proof}
Extensive numerical searches indicated that there might be no complex equiangular $(8,3)$-frames \cite{JT} (see also \cite[p.~67]{Z}). We confirm this conjecture for the first time.
\begin{theorem}\label{ne83}
There do not exist complex equiangular $(8,3)$ and $(8,5)$-frames.
\end{theorem}
\begin{proof}
We have $667$ polynomials in $12$ variables $\alpha, x_{2,3}, \hdots, x_{5,6}$, and $u$. By calculating a Gr\"obner basis, we find that $\mathcal{I}=\left\langle 1\right\rangle$, where $\mathcal{I}$ is the ideal, generated by these polynomials. Therefore there exists no complex solutions. The case $(8,5)$ follows from Theorem~\ref{naimark}.
\end{proof}
Since computing a Gr\"obner basis for the previous proof took about an hour and required about $24$ GBs of memory\footnote{The computation took about 16 hours in Magma with the argument {\ttfamily Homogenize} \normalfont set to {\ttfamily false}\normalfont.}, it seems unlikely that there is a ``trivial'' reason for nonexistence.
\begin{rem}
We remark that there exist a configuration of $6$ complex equiangular lines with common angle $\alpha=\alpha_{8,3}=\sqrt{5/21}$. This configuration, however, neither satisfies condition \eqref{12} nor the necessary conditions described by Theorem \ref{T}, and hence, cannot be extended to an equiangular $(8,3)$-frame. For the reader's amusement we briefly describe such a matrix as follows. With the notations introduced in Section 3, let $x_{2,3},\hdots,x_{4,6}$ be some (carefully chosen) roots of $625 + 1020 v^2 + 806 v^4 + 1020 v^6 + 625 v^8$, while $x_{5,6}$ be some (once again, carefully chosen) root of $15625 - 39780 v^2 + 52406 v^4 - 39780 v^6 + 15625 v^8$. Then, the self-adjoint $6\times 6$ matrix $H$ is of rank $3$, and all of its off-diagonal entries have modulus $\sqrt{5/21}$.
\end{rem}
Now we turn to the case of equiangular $(9,3)$-frames. By equation \eqref{mg}, we have $GG^\ast=3G$, thus the matrix $H:=3I-2G$ has unimodular entries and satisfies the orthogonality conditions: $HH^\ast=9I$. Hence $H$ is a self-adjoint complex Hadamard matrix with constant diagonal entries $1$. Conversely, any such complex Hadamard matrix leads to an equiangular tight frame \cite{CG}, \cite{FSZframehad}. It is more convenient to work with complex Hadamard matrices, because the orthogonality conditions are more transparent than the frame condition \eqref{mg}. Let $\omega:=-1/2+\mathbf{i}\sqrt{3}/2$ be the principal cubic root of unity once and for all.
\begin{example}[see \cite{FSZframehad}, \mbox{\cite[p.~61]{Z}}]\label{45}
The following is a one-parameter family of self-adjoint complex Hadamard matrices of order $9$ with constant diagonal $1$$:$
\[H_9^{(1)}(a)=\left[
\begin{array}{ccc|ccc|ccc}
 1 & 1 & 1 & 1 & 1 & 1 & 1 & 1 & 1 \\
 1 & 1 & 1 & \omega & \omega & \omega & \omega^2 & \omega^2 & \omega^2 \\
 1 & 1 & 1 & \omega^2 & \omega^2 & \omega^2 & \omega & \omega & \omega \\
\hline
 1 & \omega^2 & \omega & 1 & \omega^2 & \omega & a & a \omega^2 & a \omega \\
 1 & \omega^2 & \omega & \omega & 1 & \omega^2 & a \omega^2 & a \omega & a \\
 1 & \omega^2 & \omega & \omega^2 & \omega & 1 & a \omega & a & a \omega^2 \\
\hline
 1 & \omega & \omega^2 & \overline{a} & \overline{a}\omega & \overline{a}\omega^2 & 1 & \omega & \omega^2 \\
 1 & \omega & \omega^2 & \overline{a}\omega & \overline{a}\omega^2 & \overline{a} & \omega^2 & 1 & \omega \\
 1 & \omega & \omega^2 & \overline{a}\omega^2 & \overline{a} & \overline{a}\omega & \omega & \omega^2 & 1 \\
\end{array}
\right],\qquad |a|=1.\]
\end{example}
\begin{theorem}
All self-adjoint complex Hadamard matrices of order $9$ with constant diagonal $1$ belong to the one-parameter family $H_9^{(1)}(a)$, described above in Example~\ref{45}.
\end{theorem}
The proof is different from the approach of Section~\ref{s3}, as we explore $6\times 6$ submatrices.
\begin{proof}
Let $H$ be a self-adjoint complex Hadamard matrix with constant diagonal $1$. We may assume that the first row and column of $H$ is normalized to $1$. Since $H$ is an orthogonal matrix, from $\mathrm{Tr}(H)=9$ we infer that the spectrum of it is $\{[-3]^3,[3]^6\}$, and hence $H-3I$ has rank $3$. In particular, all $4\times 4$ minors of $H-3I$ must vanish. Now let us consider $K$, the leading $6\times 6$ principal submatrix of $H-3I$. We claim that $K$ must have the same structure as the principal submatrices of $H_9^{(1)}$ displayed in Example~\ref{45}: namely, if an off-diagonal entry $x_{2,3}$ is not a cubic root of unity, then there must be another off diagonal entry, in the same row of $K$, say $x_{2,4}$. Moreover, $x_{2,4}=\omega x_{2,3}$ or $x_{2,4}=\omega^2 x_{2,3}$, and hence $x_{2,3}^2+x_{2,3}x_{2,4}+x_{2,4}^2=0$. To see this, let $M_1$, $M_2$, $\hdots$, $M_{225}$ be the scaled $4\times 4$ minors of $K$, and consider the following system of polynomial equations in $11$ variables $u$, $x_{2,3},\hdots,x_{5,6}$:
\[\left.\begin{array}{rcc}
\mathrm{det}(M_1)&=&0\\
\mathrm{det}(M_2)&=&0\\
&\vdots&\\
\mathrm{det}(M_{225})&=&0\\
u(x_{2,3}^3-1)(x_{2,4}^3-1)(x_{2,3}^2+x_{2,3}x_{2,4}+x_{2,4}^2)\prod_{2\leq i<j\leq 6} x_{i,j}-1&=&0
\end{array}\right\},\]
containing the assumption that neither $x_{2,3}$ nor $x_{2,4}$ are a cubic root of unity, and they are not related by $x_{2,3}^2+x_{2,3}x_{2,4}+x_{2,4}^2$. By computing a Gr\"obner basis in about $5$ hours, utilizing nearly $30$ GBs of memory\footnote{The same computation took about 34 hours in Magma, using significantly less memory.}, we readily see that the system of equations above has no solutions. Therefore, the entries of the complex Hadamard matrix $H$ are either cubic roots of unity, or if not, then these non cubic entries must be (by orthogonality) $a$, $a\omega$ and $a\omega^2$ for some complex unimodular number $a$, all three of them appearing exactly once in some, say the second, row of $H$.

Now assume that $H$ contains a non cubic entry $a$. We may assume, up to equivalence, that the entries in the first row of $H$ are all $1$, while its second row is $H_2:=[1 , 1 , a , a\omega , a\omega^2 , \omega , \omega , \omega^2 , \omega^2]$. Since $H$ is self-adjoint, we have $H_{3,1}=H_{3,3}=1$, $H_{3,2}=\overline{a}$, and hence the third row must be some permutation of the second, by replacing $a\rightarrow \overline{a}$ in it. By examining all possible permutations (by computers), we readily see that if $a\neq -1$, then there are four possibilities for the third row, all of them leading to $[1,\overline{a},1,\omega^2,\omega,\overline{a}\omega^2,\omega^2,\overline{a}\omega,\omega]$ after permuting the last four columns, if necessary. It follows that the first five rows of $H$ are equivalent to
\[H_5(a):=\left[\begin{array}{ccccccccc}
1 & 1 & 1 & 1 & 1 & 1 & 1 & 1 & 1\\
1 & 1 & a & a\omega & a\omega^2 & \omega & \omega & \omega^2 & \omega^2\\
1 & \overline{a} & 1 & \omega^2 & \omega & \overline{a}\omega^2 & \omega^2 & \overline{a}\omega & \omega\\
1 & \overline{a}\omega^2 & \omega & 1 & \omega^2 & \overline{a}\omega & \omega^2 & \overline{a} & \omega\\
1 & \overline{a}\omega & \omega^2 & \omega & 1 & \overline{a} & \omega^2 & \overline{a}\omega^2 & \omega\\
\end{array}\right].\]
Since $H$ is self-adjoint, the first $5$ columns are already known, and the unknown entries in the sixth and eighth rows and columns must be some cubic root of unity. Hence all remaining entries are some cubic root of unity, and we arrived to the solution displayed in Example~\ref{45}.

If $a=-1$, then there are $5$ ways to extend the second row with a third, up to permuting the last four columns. However, only one of these can be extended (in a unique way) by further orthogonal rows, leading to the matrix $H_9^{(1)}(-1)$ of Example~\ref{45}.

Finally, if $H$ is composed of cubic roots of unity, then it is easy to see that the first three rows must be the same as in the matrix $H_9^{(1)}(1)$. This $3\times 9$ matrix can be extended by an orthogonal row in $12$ distinct ways, but all of these are permutation equivalent to the first four rows of $H_9^{(1)}(1)$. Now it is elementary to fill out the missing entries and get $H_9^{(1)}(1)$.
\end{proof}
\begin{corollary}
If $G$ is the Gram matrix of an equiangular $(9,3)$-frame, then $G$ belongs to the family $G_9^{(1)}(a):=(3I_9-H_9^{(1)}(a))/2$, where the matrix $H_9^{(1)}(a)$ is given in Example~\ref{45}.
\end{corollary}
It is worthwhile to note that a self-adjoint complex Hadamard matrix of order $9$ does not necessarily have constant diagonal. Indeed, by replacing in the matrix $H_9^{(1)}(a)$ its lower right $3\times 3$ submatrix with its negative, we get a self-adjoint complex Hadamard matrix with non constant diagonal \cite{FSZframehad}. These matrices do not correspond to equiangular tight frames.

We believe that the combination of our ideas with a deeper understanding of the polynomial ideals, generated by minors of matrices \cite{AC}, and with more powerful computers, will lead to further classification results of this kind on equiangular tight frames in $\mathbb{C}^d$ for $d\geq 4$.
\section*{Acknowledgement}
Work on this paper began during the AIM workshop ``Frame theory intersects geometry'' in Palo Alto. We thank the organizers for their kind invitation. We are also grateful to Markus Grassl for some helpful comments regarding questions raised during the preparation of this manuscript. This work was supported by the Hungarian National Research Fund OTKA $K$-$77748$ and by the JSPS KAKENHI Grant Number $24\cdot02807$.

\end{document}